\documentclass[a4paper,12pt]{amsart}
\newif\ifprivate
\ifprivate
\usepackage{svn}
\SVN $Revision: 28 $
\SVN $Date: 2008-04-04 15:18:25 +0200 (Fre, 04 Apr 2008) $
\fi
\usepackage{a4wide,url,ifthen,graphicx}
\usepackage{hyperref} 
\usepackage{algorithm,algorithmic}

\newcommand{\assign}{\leftarrow}
\newcommand{\Binary}{\mathop{\mathsf{Binary}}}
\newcommand{\bfc}{\mathbf{c}}
\newcommand{\bfd}{\mathbf{d}}
\newcommand{\bfn}{\mathbf{n}}
\newcommand{\bfnull}{\mathbf{0}}
\newcommand{\bfq}{\mathbf{q}}
\newcommand{\bfr}{\mathbf{r}}
\newcommand{\calC}{\mathcal{C}}
\newcommand{\downto}{\textbf{downto}}
\renewcommand{\MR}[1]{}
\newcommand{\N}{\mathbb{N}}
\newcommand{\NAF}{\mathop{\mathsf{NAF}}}
\newcommand{\SJSF}{\mathop{\mathsf{SJSF}}}
\newcommand{\val}{\mathop{\mathsf{value}}}
\newcommand{\weight}{\mathop{\mathsf{weight}}\nolimits}
\newcommand{\zp}[2]{\genfrac{}{}{0pt}{}{#1}{#2}}
\newcommand{\zeros}{\mathop{\mathsf{zeros}}\nolimits}

\newcommand{\klammern}[4][]%
{\ifthenelse{\equal{#1}{}}{\left#2}{\csname#1\endcsname#2}%
#4\ifthenelse{\equal{#1}{}}{\right#3}{\csname#1\endcsname#3}}

\newcommand{\abs}[2][]{\klammern[#1]{\lvert}{\rvert}{#2}}
\newcommand{\expect}[2][]{\mathbb{E}\klammern[#1]{(}{)}{#2}}
\newcommand{\floor}[2][]{\klammern[#1]{\lfloor}{\rfloor}{#2}}
\newcommand{\prob}[2][]{\mathbb{P}\klammern[#1]{(}{)}{#2}}

\newtheorem{theorem}{Theorem}
\newtheorem{lemma}{Lemma}[section]
\theoremstyle{definition}
\newtheorem{definition}[lemma]{Definition}
\newtheorem{corollary}[lemma]{Corollary}
\theoremstyle{remark}
\newtheorem{remark}[lemma]{Remark}

\newif\ifpdf
\ifx\pdfoutput\undefined
    \pdffalse          
\else
\ifthenelse{\equal{\the\pdfoutput}{1}}{\pdftrue}{\pdffalse}
\fi

\ifpdf
\fi

\title[Complements and Signed Digit Representations]{Complements and Signed
  Digit Representations: Analysis of a Multi-Exponentiation-Algorithm
  of Wu, Lou, Lai and Chang}
\author{Clemens Heuberger}
\address{Institut f\"ur Mathematik B\\Technische Universit\"at Graz\\Austria}
\email{clemens.heuberger@tugraz.at}
\thanks{This paper was written while C.~Heuberger 
was a visitor at the
  Center of Experimental Mathematics at the University of Stellenbosch. He
  thanks the center
 for its hospitality. He is also supported
by the Austrian
Science Foundation FWF, project S9606, that is part of the
Austrian National Research Network ``Analytic Combinatorics
and Probabilistic Number Theory.''}
\author{Helmut Prodinger}
\address{Department of Mathematics\\University of Stellenbosch\\South Africa}
\email{hproding@sun.ac.za}
\thanks{H.~Prodinger is supported by the NRF grant 2053748 
of the South African National Research Foundation and by
the Center of Experimental Mathematics of the University of Stellenbosch.}


\begin{document}
\begin{abstract}
  Wu, Lou, Lai and Chang proposed a multi-exponentiation algorithm using binary complements and the non-adjacent form. The purpose of this paper is to show that neither the analysis of the algorithm given by its original proposers nor that by other authors are correct. In fact it turns out that the complement operation does not have significant influence on the performance of the algorithm and can therefore be omitted.
\end{abstract}

\maketitle
\ifprivate \thispagestyle{headings}\pagestyle{headings} \markboth{\jobname{}
rev. \SVNRevision{} ---
  \SVNDate{} \SVNTime}{\jobname{} rev. \SVNRevision{} --- \SVNDate{} \SVNTime}
\fi

\section{Introduction}
An efficient way to compute
a power  $a^n$ is to use the binary expansion $\sum_j d_j2^j$ of $n$ and
compute $a^n$ by a square and multiply algorithm~\cite{Knuth:1998:Art:2},
\begin{equation*}
  a^{\sum_{j=0}^{\ell-1} d_j2^j}=(((a^{d_\ell-1})^2\cdot a^{d_{\ell-2}})^2\cdots a^{d_1})^2a^{d_0},
\end{equation*}
where the number of squarings needed is  $\ell-1$, whereas the number of
multiplications by $a^{d_j}$ equals to the number of nonzero $d_j$ minus $1$ (under the
assumption that $d_{\ell-1}=1$), because multiplications with $a^0$ can be
omitted. Among the various possible optimisations is the use of signed digit
representations, i.e., allowing digits $-1$ also, which results in
multiplications by $a^{-1}$. This is of particular interest if $a^{-1}$ is
known or can be computed easily, e.g., in the point group of an elliptic curve.

Some cryptosystems also need multi-exponentiation $\prod_{j=1}^D a_j^{n_j}$
(usually for $D\in\{2,3\}$). A trivial approach would be to compute $a_j^{n_j}$ separately
for $j\in\{1,\ldots,D\}$ and multiply the results, however,
Straus\footnote{This approach is frequently called
  Shamir's trick, we refer to \cite{Bernstein:2002:expon-algor} for a
  discussion of this attribution. Similar suggestions have been made in
  \cite{Pekmestzi:1989:compl-number-multip} and
  \cite{Dimitrov-Jullien-Miller:2000:compl-fast-algor-multiex}.}~\cite{Straus:1964:addit}
 demonstrated that an interleaved approach
leads to better results. For simplicity of exposition, we restrict ourselves to
$D=2$ at this point, although the method works for arbitrary $D$. For computing
$a^mb^n$, we take binary expansions $m=\sum_{j=0}^{\ell-1}c_j 2^j$ and
$n=\sum_{j=0}^{\ell-1}d_j 2^j$ and compute 
\begin{equation*}
  a^{\sum_{j=0}^{\ell-1}c_j 2^j}b^{\sum_{j=0}^{\ell-1}d_j
    2^j}=(((a^{c_{\ell-1}}b^{d_{\ell-1}})^2a^{c_{\ell-2}}b^{d_{\ell-2}})^2\ldots a^{c_{1}}b^{d_{1}})^2a^{c_{0}}b^{d_{0}}.
\end{equation*}
If $a^cb^d$ are precomputed for all admissible pairs of digits $(c,d)$, then
the number of squarings equals $\ell-1$ and the number of multiplications
by $a^{c_j}b^{d_j}$ equals the joint Hamming weight, i.e., the number of pairs $(c_j,d_j)\neq (0,0)$, minus one (under the
assumption that $(c_{\ell-1},d_{\ell-1})\neq (0,0)$) plus the time needed for
the precomputation, which is clearly constant and does not depend on the length
of the expansion. In dimension $D=2$, pairs of integers can be identified with
complex numbers as proposed in \cite{Pekmestzi:1989:compl-number-multip}, but
this is merely an other way to formulate the procedure. Allowing negative digits again, redundancy can be used to
decrease the joint Hamming weight.

As in the case of dimension $1$, there is a syntactic condition which yields
expansions of minimal joint Hamming weight, cf.\
\cite{Solinas:2001:low-weigh},
\cite{Grabner-Heuberger-Prodinger:2004:distr-results-pairs},
\cite{Proos:2003:joint-spars}. In dimension $D=2$, it is shown that these
optimal expansions have expected Hamming weight $(1/2)\ell+O(1)$, so that the
total expected number of multiplications equals\footnote{The authors of
  \cite{Wu-Lou-Lai-Chang:2007:fast-modul} and
  \cite{Sun-Huai-Sun-Zhang:2007:comput-effic} erroneously write $1.503\ell$ without further
comment.} $(3/2)\ell+O(1)$, cf.\ also
\cite{Avanzi:2005:multi-expon-crypt} and
\cite{Grabner-Heuberger-Prodinger-Thuswaldner:2005:analy-linear}. We refer
to \cite{Heuberger-Muir:2007:minim-weigh} for a more detailed introduction
with more references.

The authors of \cite{Wu-Lou-Lai-Chang:2007:fast-modul} present an
alternative approach in dimension $D$ involving the complement of the binary expansion and claim
that the expected number of multiplications of their algorithm equals
$1.304\ell+O(1)$. This was followed by
\cite{Sun-Huai-Sun-Zhang:2007:comput-effic} whose authors claim to correct the
result \cite{Wu-Lou-Lai-Chang:2007:fast-modul} and that the same
algorithm needs $1.471\ell+O(1)$ multiplications on average. The purpose of this note is to
show that both results are incorrect. 
We explain why the result
cannot be better than the above  optimal joint expansions
(Theorem~\ref{th:extended-optimality}), show that the algorithm essentially
corresponds to taking the NAF for both arguments
(Theorem~\ref{th:weight-difference}) and give the correct expected number
$(14/9)\ell+O(1)=1.555\ldots\ell+O(1)$ of multiplications (Theorem~\ref{th:markov-chain-analysis}). 

In Section~\ref{sec:digit-expansions}, we collect notations and well-known
results on digit expansions. Section~\ref{sec:algorithm} presents the algorithm
proposed by \cite{Wu-Lou-Lai-Chang:2007:fast-modul}, which is analysed in
Section~\ref{sec:analysis}. Finally, in Section~\ref{sec:errors}, we discuss
where the errors in the probabilistic arguments of
\cite{Wu-Lou-Lai-Chang:2007:fast-modul}  and
\cite{Sun-Huai-Sun-Zhang:2007:comput-effic} lie.

\section{Digit Expansions}\label{sec:digit-expansions}
\subsection{Digit Expansions of Integers}
A signed digit expansion of an integer $n$ is a word $d_{\ell-1}\ldots d_0$
over the alphabet $\{-1,0,1\}$ such that $n=\val(d_{\ell-1}\ldots
d_0)=\sum_{j=0}^{\ell-1}d_j2^j$. The \emph{(Hamming) weight}
$\weight(d_{\ell-1}\ldots d_0 )$ of $d_{\ell-1}\ldots d_0$ is the number of
non-zero digits $d_j$.

When all digits are in $\{0,1\}$, we speak of the \emph{standard binary
  expansion} of $n$ (which must then be non-negative). The standard binary
expansion of $n$ is denoted by $\Binary(n)$.

While every integer $n$ admits infinitely many
signed digit expansions, one special expansion has attracted particular
attention.

\begin{definition}
  A signed digit expansion $d_{\ell-1}\ldots d_0$ is called a
  \emph{Non-Adjacent-Form (NAF)}, if $d_jd_{j+1}=0$ for all $j$, i.e., there
  are no adjacent non-zero digits.
\end{definition}
Reitwiesner~\cite{Reitwiesner:1960} showed that every integer $n$ admits a unique
NAF, denoted by $\NAF(n)$,  and that $\NAF(n)$ minimises the Hamming weight over all signed digit
expansions of $n$. The NAF is known under various names, e.g., the
\emph{canonical signed digit expansion}.

The ones' complement of a standard binary expansion $d_{\ell-1}\ldots d_0$ is
$\widehat{d_{\ell-1}}\ldots \widehat{d_0}$, where 
\begin{equation*}
  \widehat{d}=1-d=
  \begin{cases}
    0&d=1,\\
    1&d=0.
  \end{cases}
\end{equation*}
It is immediate from the definition that
\begin{equation*}
  \val(d_{\ell-1}\ldots d_0)=2^{\ell}-\val(\widehat{d_{\ell-1}}\ldots
  \widehat{d_0})-1.
\end{equation*}
This can be seen as another signed digit expansion,
\begin{equation*}
  \val(d_{\ell-1}\ldots
  d_0)=\val(1(-\widehat{d_{\ell-1}})\ldots(-\widehat{d_1})(-\widehat{d_0}-1)),
\end{equation*}
with the exception that the least significant digit is now in $\{-1,-2\}$. 

\subsection{Digit Expansion of Vectors}
A signed digit joint expansion of an integer vector $(\zp mn)$ is a word
$\bfd^{(\ell-1)}\ldots\bfd^{(0)}$ over the alphabet $\{-1,0,1\}^2$ such that
\begin{equation*}
  \left(\zp mn\right)=\val(\bfd^{(\ell-1)}\ldots\bfd^{(0)})=\sum_{j=0}^{\ell-1} \bfd_j 2^j.
\end{equation*}
The \emph{joint (Hamming) weight} of $\bfd^{(\ell-1)}\ldots\bfd^{(0)}$ is the number
of $j$ with $\bfd^{(j)}\neq \bfnull:=(\zp 00)$. The components of $\bfd^{(j)}$
are written as $(\zp{d^{(j)}_1}{d^{(j)}_2})$.

So, the digits are now column vectors. One simple way to obtain such a joint
expansion is to independently choose two signed-binary expansions
$d^{(\ell-1)}_1\ldots d^{(0)}_1$ and $d^{(\ell-1)}_2\ldots d^{(0)}_2$ of $m$
and $n$, respectively, and to write them on top of each other. In order to
achieve small joint weight, one might take the NAFs of $m$ and $n$ and write
them on top of each other; the expected joint weight is $(5/9)\ell+O(1)$,
cf.\ \cite{Grabner-Heuberger-Prodinger-Thuswaldner:2005:analy-linear}.\footnote{The
heuristic argument to see this would be that the expected number of zeros in a NAF is
$(2/3)\ell+O(1)$, thus the expected number of digits vectors $\bfnull$ when
writing two NAFs on top of each other should be $(4/9)\ell+O(1)$, which leaves
$(5/9)\ell+O(1)$ for the Hamming weight. As we shall see when discussing the
errors in \cite{Wu-Lou-Lai-Chang:2007:fast-modul} and
\cite{Sun-Huai-Sun-Zhang:2007:comput-effic}, such arguments do not take
possible dependence of the digits into account and might lead to
errors. Therefore, we refer to the precise analysis in
\cite{Grabner-Heuberger-Prodinger-Thuswaldner:2005:analy-linear}.}

Solinas~\cite{Solinas:2001:low-weigh} discussed the ``Joint Sparse Form'', a
joint expansion which minimises the joint weight over all joint expansions of
the same pair of integers. Here, we use a simplified version introduced in
\cite{Grabner-Heuberger-Prodinger:2004:distr-results-pairs}.
\begin{definition}
  A joint expansion $\bfd^{(\ell-1)}\ldots\bfd^{(0)}$ is called a \emph{Simple
    Joint Sparse Form (SJSF)}, if the following two conditions hold for all
  $j\ge 0$:
  \begin{gather}
    \text{If }\abs[normal]{d^{(j)}_1}\neq\abs[normal]{d^{(j)}_2}\text{, then }\abs[normal]{d^{(j+1)}_1}=\abs[normal]{d^{(j+1)}_2},\label{eq:simple-joint-sparse-form:syntax-10}\\
    \text{If }\abs[normal]{d^{(j)}_1}=\abs[normal]{d^{(j)}_2}=1\text{, then } \bfd^{(j+1)}=\bfnull.\label{eq:simple-joint-sparse-form:syntax-11}
  \end{gather}
\end{definition}
In \cite{Grabner-Heuberger-Prodinger:2004:distr-results-pairs}, we proved that
every pair of integers $(\zp mn)$ admits exactly one SJSF, denoted by
$\SJSF(\zp mn)$, and that $\SJSF(\zp mn)$
minimises the joint weight over all joint expansions of $(\zp mn)$ with digits in
$\{-1,0,1\}$. In
\cite{Grabner-Heuberger-Prodinger-Thuswaldner:2005:analy-linear}, it was shown
that the expected joint weight of a
SJSF of length $\ell$ is $(1/2)\ell+O(1)$.

\section{Multi-Exponentiation Algorithm}\label{sec:algorithm}
We now present the algorithm of \cite{Wu-Lou-Lai-Chang:2007:fast-modul} in
Algorithm~\ref{alg:Wu-Lou-Lai-Chang}. We assume that $a_1^{d_1} a_2^{d_2}$ have
been precomputed for $(d_1,d_2)\in \{-1,0,1\}^2$ and are used in
Lines~\ref{line:multiply-most-significant-digits} and \ref{line:multiply}. It
might happen that $d_k^{(0)}=-2$, in that case, two multiplications are needed
in Line~\ref{line:multiply}. Note that the length of the NAF may exceed the
length of the standard binary expansion by at most $1$.

\begin{algorithm}[htbp]
  \begin{algorithmic}[1]
    \REQUIRE $a_1$, $a_2\in G$ (some Abelian group), $n_1$, $n_2\in \N$
    \ENSURE $b=a_1^{n_1}a_2^{n_2}$
    \STATE $\ell=\floor{\log_2(\max(n_1,n_2))}+1$
    \FOR{$k=1,2$}
      \STATE $b^{(\ell-1)}_k\ldots b^{(0)}_k \assign \Binary(n_k)$
      \IF{$\weight(b^{(\ell-1)}_k\ldots b^{(0)}_k)>\ell/2$}
      \STATE $(d^{(\ell)}_k\ldots d^{(0)}_k)\assign \NAF(\val((b^{(\ell-1)}_k-1)\ldots (b^{(0)}_k-1)))$
      \STATE $d^{(\ell)}_k \assign d^{(\ell)}_k+1$
      \STATE $d^{(0)}_k \assign d^{(0)}_k-1$
      \ELSE
      \STATE $(d^{(\ell)}_k\ldots d^{(0)}_k)\assign \NAF(n_k)$
      \ENDIF
      \STATE \COMMENT{We have $n_k=\val(d^{(\ell)}_k\ldots d^{(0)}_k)$}
    \ENDFOR\label{line:recoding-completed}
    \STATE $b\assign a_1^{d_1^{(\ell)}}a_2^{d_2^{(\ell)}}$ \label{line:multiply-most-significant-digits}
    \FOR{$j=\ell-1$ \downto{} $0$}
      \STATE $b\assign b^2 a_1^{d_1^{(j)}}a_2^{d_2^{(j)}}$ \label{line:multiply}
      \STATE \COMMENT{We have $b=a_1^{\sum_{k=j}^\ell d_1^{(k)}2^{k-j}}a_2^{\sum_{k=j}^\ell d_2^{(k)}2^{k-j}}$}
    \ENDFOR
  \end{algorithmic}
  \caption{Wu, Lou, Lai and Chang's~\cite{Wu-Lou-Lai-Chang:2007:fast-modul}
    Algorithm for Multi-Exponentiation}
  \label{alg:Wu-Lou-Lai-Chang}
\end{algorithm}

The idea of the algorithm is the following: If the weight of the binary
expansion of the exponent $n_j$ is large ($>\ell/2$), then $n_j$ is represented
by its complement and the NAF of the complement is used. The heuristic is that
reducing the weight of the expansion before converting it to its NAF should
result in a lower weight of the NAF.

Note that after execution of Line~\ref{line:recoding-completed} of
Algorithm~\ref{alg:Wu-Lou-Lai-Chang}, we have a joint expansion
$\bfd^{(\ell)}\ldots\bfd^{(0)}$ of $\mathbf{n}=(\zp{n_1}{n_2})$ where
$\bfd^{(j)}\in\{-1,0,1\}^2$ for $j>0$ and $\bfd^{(0)}\in\{-2,-1,0,1\}^2$. The number of
group multiplications is $\ell$ (for the squarings) plus
\begin{equation*}
  \weight_1(\bfd^{(\ell)}\ldots\bfd^{(0)}):=\sum_{j=0}^{\ell}\max\{\abs[normal]{d^{(j)}_k} : k\in\{1,2\}\}
\end{equation*}
minus $1$ (no multiplication is required for the most significant digit).
Note that for expansions with digits $\{-1,0,1\}$, the notions of $\weight_1$
and $\weight$ agree. 

\section{Analysis of the Algorithm}\label{sec:analysis}

We will now extend the optimality proof for the SJSF from 
\cite{Grabner-Heuberger-Prodinger:2004:distr-results-pairs} to the case of
digits from $\{-2,-1,0,1,2\}$. It turns out that we can allow arbitrary digits
of absolute value at most $2$ without changing the result. In fact, we do not
even need any particular properties of the SJSF.

\begin{theorem}\label{th:extended-optimality}
  Let $\bfd^{(\ell-1)}\ldots\bfd^{(0)}$ be a word over the alphabet
  $\{-2,-1,0,1,2\}^2$ and $\bfn=\val(\bfd^{(\ell-1)}\ldots\bfd^{(0)})$. Then we
  have
  \begin{equation*}
    \weight_1(\bfd^{(\ell-1)}\ldots\bfd^{(0)})\ge
    \weight_1(\SJSF(\bfn))=\weight(\SJSF(\bfn)),
  \end{equation*}
  i.e., $\SJSF(\bfn)$ minimises
  $\weight_1$ over all expansions of $\bfn$ with digits in $\{-2,-1,0,1,2\}$.
\end{theorem}
Before proving the theorem, we note that this already shows that the analysis
in \cite{Wu-Lou-Lai-Chang:2007:fast-modul} and
\cite{Sun-Huai-Sun-Zhang:2007:comput-effic} cannot be correct:

\begin{corollary}\label{cor:at-least-3-2}
  The expected number of group multiplications needed by
  Algorithm~\ref{alg:Wu-Lou-Lai-Chang} is at least $(3/2)\ell+O(1)$.
\end{corollary}
\begin{proof}[Proof of Corollary~\ref{cor:at-least-3-2}]
  For every $\bfn$, the number of multiplications used by
  Algorithm~\ref{alg:Wu-Lou-Lai-Chang} is not less then the number of
  multiplications needed when using the SJSF, which is known to be
  $(3/2)\ell+O(1)$ from \cite{Grabner-Heuberger-Prodinger-Thuswaldner:2005:analy-linear}.
\end{proof}

The essential step in the proof of Theorem~\ref{th:extended-optimality} is the
following lemma.

\begin{lemma}\label{le:digit-2-reduction}
  Let $\bfd^{(\ell-1)}\ldots\bfd^{(0)}$ be a word over the alphabet
  $\{-2,-1,0,1,2\}^2$ where $k>0$ digit vectors contain a digit of absolute value $2$. Then there is a
  word $\bfc^{(\ell'-1)}\ldots\bfc^{(0)}$ over the alphabet $\{-2,-1,0,1,2\}^2$ with less than $k$ digit vectors
  containing a digit of
  absolute value $2$, 
  $\val(\bfd^{(\ell-1)}\ldots\bfd^{(0)})=\val(\bfc^{(\ell'-1)}\ldots\bfc^{(0)})$ and
  $\weight_1(\bfd^{(\ell-1)}\ldots\bfd^{(0)})\ge\weight_1(\bfc^{(\ell'-1)}\ldots\bfc^{(0)})$.
\end{lemma}
\begin{proof}
  We prove the lemma by induction on
  $\weight_1(\bfd^{(\ell-1)}\ldots\bfd^{(0)})$.
  Choose $j$ maximal such that $\bfd^{(j)}$ contains a digit of absolute
  value $2$. We write $\bfd^{(j)}=2\bfq+\bfr$ with $\bfq\in \{-1,0,1\}^2$ and
  $\bfr\in \{0,1\}^2$. We have 
  $\weight_1(\bfd^{(\ell-1)}\ldots\bfd^{(j+2)}(\bfd^{(j+1)}+\bfq))\le
  \weight_1(\bfd^{(\ell-1)}\ldots\bfd^{(0)})-1$ and
  $\bfd^{(\ell-1)}\ldots\bfd^{(j+2)}(\bfd^{(j+1)}+\bfq)$ is an expansion with
  digits from $\{-2,-1,0,1,2\}$, where digits of absolute value $2$ can only occur
  in $(\bfd^{(j+1)}+\bfq)$. Thus, by induction hypothesis, there is an
  expansion $\bfc^{(\ell'-1)}\ldots \bfc^{(j+1)}$ with digits from
  $\{-1,0,1\}$, $\val(\bfc^{(\ell'-1)}\ldots
  \bfc^{(j+1)})=\val(\bfd^{(\ell-1)}\ldots\bfd^{(j+2)}(\bfd^{(j+1)}+\bfq))$ and
  \begin{equation*}
    \weight_1(\bfc^{(\ell'-1)}\ldots
    \bfc^{(j+1)})\le
    \weight_1(\bfd^{(\ell-1)}\ldots\bfd^{(j+2)}(\bfd^{(j+1)}+\bfq)).
  \end{equation*}
  Setting
  $\bfc^{(j)}\bfc^{(j-1)}\ldots\bfc^{(0)}=\bfr \bfd^{(j-1)}\ldots
  \bfd^{(0)}$, we see that 
  \begin{align*}
    \weight_1(\bfc^{(\ell'-1)}\ldots \bfc^{(0)})&\le
    \weight_1(\bfd^{(\ell-1)}\ldots\bfd^{(j+2)}(\bfd^{(j+1)}+\bfq))+
    \weight_1(\bfr \bfd^{(j-1)}\ldots \bfd^{(0)})\\
    &\le
    \weight_1(\bfd^{(\ell-1)}\ldots\bfd^{(j+2)}\bfd^{(j+1)})+1+1+
    \weight_1(\bfd^{(j-1)}\ldots \bfd^{(0)})\\
    &=\weight_1(\bfd^{(\ell-1)}\ldots\bfd^{(0)})
  \end{align*}
  and that $\bfc^{(\ell'-1)}\ldots \bfc^{(0)}$ satisfies the requirements of the lemma.
\end{proof}

We are now able to prove Theorem~\ref{th:extended-optimality}.

\begin{proof}[Proof of Theorem~\ref{th:extended-optimality}]
  Repeated application of Lemma~\ref{le:digit-2-reduction} shows that 
  there  is an
  expansion $\bfc^{(\ell'-1)}\ldots\bfc^{(0)}$ with digits from $\{-1,0,1\}$ with 
  $\val(\bfc^{(\ell-1)}\ldots\bfc^{(0)})=\bfn$ and
  \begin{equation*}
    \weight_1(\bfd^{(\ell-1)}\ldots\bfd^{(0)})\ge\weight_1(\bfc^{(\ell-1)}\ldots\bfc^{(0)}).  
  \end{equation*}

  Taking into account that
  $\weight_1(\bfc^{(\ell-1)}\ldots\bfc^{(0)})=\weight(\bfc^{(\ell-1)}\ldots\bfc^{(0)})$ and the optimality of the SJSF \cite[Theorem~2]{Grabner-Heuberger-Prodinger:2004:distr-results-pairs} completes the proof of the theorem.
\end{proof}

The next question is how Algorithm~\ref{alg:Wu-Lou-Lai-Chang} compares with the
simple strategy of directly using $(\zp{\NAF(n_1)}{\NAF(n_2)})$. 

\begin{theorem}\label{th:weight-difference}
  Let $b_{\ell-1}\ldots b_0$ be a standard binary expansion and
  $\widehat{b_{\ell-1}}\ldots \widehat{b_0}$ its complement. Then
  \begin{equation*}
    \abs[normal]{\weight(\NAF(\val(b_{\ell-1}\ldots
      b_0)))-\weight(\NAF(\val(\widehat{b_{\ell-1}}\ldots \widehat{b_0})))}\le 2.
  \end{equation*}
\end{theorem}
This means that the strategy of taking the NAF of the complement of a number
at best induces a saving of $2$ in the weight, which is subsequently lost when
adding the two corrective terms. In other words, this strategy never yields a
lower weight than a direct use of the NAF.

\begin{proof}[Proof of Theorem~\ref{th:weight-difference}]
  The NAF can be computed from the standard binary expansion of a positive
  integer $n$ by a transducer automaton from right to left (cf.\
  \cite[Figure~2]{Heuberger-Prodinger:2006:analy-alter}), reproduced here as
  Figure~\ref{fig:NAF-transducer}. For typographical reasons, negative digits
  $-d$ are written as $\bar d$.

  \begin{figure}[htbp]
    \centering
    \includegraphics{muir1.1}
    \caption{Transducer for converting the standard binary expansion to the NAF.}
    \label{fig:NAF-transducer}
  \end{figure}

  In order to compare the NAFs of $n$ and its complement, we compute these NAFs
  simultaneously by one transducer. This transducer is shown in
  Figure~\ref{fig:double-NAF-transducer}. Here, $\bot$ denotes the end of the
  input and $\varepsilon$ denotes the empty word. 

\begin{figure}[htbp]
  \centering
  \includegraphics{DoubleNAFTransducer.1}
  \caption{Transducer converting a standard binary expansion $d_{\ell-1}\ldots
    d_0$ into $\left(\zp{\NAF(\val(d_{\ell-1}\ldots
        d_0))}{\NAF(\val(\widehat{d_{\ell-1}}\ldots
        \widehat{d_0}))}\right)$, i.e., the NAF and
    the NAF of the complement.}
  \label{fig:double-NAF-transducer}
\end{figure}

  The transducer reads the standard binary expansion of $n$ from right to left
  and writes vectors of digits containing the NAF of $n$ and its
  complement. The labels of the states correspond to carries, the ``binary
  point'' indicates the look-ahead, i.e., the number of digits read minus the
  number of digits written. The transducer can be decomposed in four
  strongly connected components: $\calC_1=\{\zp00\}$ (the initial state only),
  $\calC_2=\{.\zp{0}{1}, .\zp{1}0 \}$, $\calC_3=\{.\zp{0}{2},
  .\zp{1}{1}, .\zp{2}0 \}$ and $\calC_4=\{\text{terminal state} \}$.

  When leaving $\calC_1$, the weights of the two output rows are trivially
  equal. After leaving $\calC_2$, the difference of the weights is at most $1$,
  as there is only one cycle in $\calC_2$ and its weights are balanced. Within
  $\calC_3$, the weight of the output in both rows is always the same. When leaving
  $\calC_3$, another difference of at most $1$ might occur. Summing up, the
  weight difference is at most $2$.
\end{proof}

We can now quantify the performance of Algorithm~\ref{alg:Wu-Lou-Lai-Chang}:

\begin{theorem}\label{th:markov-chain-analysis}
  For a random pair $(n_1,n_2)$ with $0\le n_1,n_2<2^\ell$ (all of these pairs
  are considered to be equally likely), the expected number
  of multiplications when executing Algorithm~\ref{alg:Wu-Lou-Lai-Chang} is
  \begin{equation*}
    \frac{14}9\ell+O(1)=1.555\ldots\ell+O(1).
  \end{equation*}
\end{theorem}
\begin{proof}
  Before considering pairs, it is essential to understand the effect of
  Algorithm~\ref{alg:Wu-Lou-Lai-Chang} on a single integer, which is encoded by
  the transducer automaton in Figure~\ref{fig:double-NAF-transducer}.
  We number the states of the transducer as follows:
  \begin{equation*}
    \begin{array}{c|*{6}{c}}
      \text{number}&1&2&3&4&5&6\\\hline
      \text{state}\rule{0pt}{14pt} &\zp{0}{0}& .\zp{0}{1}& .\zp{1}{0}& .\zp{0}{2}& .\zp{2}{0}& .\zp{1}{1}
    \end{array}
  \end{equation*}
  The transition probability matrix of the transducer is
  \begin{equation*}\renewcommand{\arraystretch}{1.2}
    P=
    \begin{pmatrix}
      0 & \frac{1}{2} & \frac{1}{2} & 0 & 0 & 0 \\
      0 & 0 & \frac{1}{2} & \frac{1}{2} & 0 & 0 \\
      0 & \frac{1}{2} & 0 & 0 & \frac{1}{2} & 0 \\
      0 & 0 & 0 & \frac{1}{2} & 0 & \frac{1}{2} \\
      0 & 0 & 0 & 0 & \frac{1}{2} & \frac{1}{2} \\
      0 & 0 & 0 & \frac{1}{2} & \frac{1}{2} & 0
    \end{pmatrix},
  \end{equation*}
  i.e., the entry in row $i$, column $j$, is the probability of a transition
  from state $i$ to state $j$. These are $0$ or $1/2$ depending on whether
  there is a transition from $i$ to $j$ at all; the digits of the standard binary
  expansion are independently uniformly distributed.

  The probability of reaching state $j$ after reading $k$ digits is the $j$th
  component of
  \begin{equation*}
    (1, 0, 0, 0, 0, 0)P^k=
 \left( 0,2^{-k},2^{-k},\frac{1}{3}+O(2^{-k}),\frac{1}{3}+O(2^{-k}),\frac{1}{3}+O(2^{-k})\right).
  \end{equation*}
  When leaving States~$4$ or $5$, the transducer writes a digit $0$, when
  leaving State~6, a non-zero digit is written. The situation in States~2 and 3
  is more complicated as it depends on the weight of the standard binary
  expansion, however, since we are in these states with probability $2^{-k}$,
  we do not have to deal with this problem. Summing up, the probability that
  the transducer writes a digit $0$ as the $(k-1)$st output digit is
  $p_{k-2}:=2/3+O(2^{-k})$.

  Let now $\bfd^{(\ell)}\ldots\bfd^{(0)}$ be the joint expansion of
  $(n_1,n_2)$ produced by Algorithm~\ref{alg:Wu-Lou-Lai-Chang}, where $n_1$ and $n_2$ are independent and uniformly distributed
  random variables on $\{0,\ldots,2^\ell-1\}$. Denote by
  $\zeros(\bfd^{(\ell)}\ldots\bfd^{(0)})$ the number of digit vectors $\bfnull$ in
  $\bfd^{(\ell)}\ldots\bfd^{(0)}$, which implies that
  $\zeros(\bfd^{(\ell)}\ldots\bfd^{(0)})=\ell+1-\weight(\bfd^{(\ell)}\ldots\bfd^{(0)})$.
  Then the expectation of $\zeros(\bfd^{(\ell)}\ldots\bfd^{(0)})$ can be computed as
  \begin{align*}
    \expect{\zeros(\bfd^{(\ell)}\ldots\bfd^{(0)})}
    &=\sum_{k=0}^{\ell}\prob{\bfd^{(k)}=\bfnull}
    = \sum_{k=0}^{\ell} \prob{d_1^{(k)}=0}\prob{d_2^{(k)}=0}\\
    &=\sum_{k=0}^{\ell} p_k^2
    =\sum_{k=0}^{\ell} \left(\frac49+O(2^{-k})\right)
    =\frac49\ell+O(1), 
  \end{align*}
  where we used the fact that the random variables $d_1^{(k)}$ and $d_2^{(k)}$
  are independent (which is a consequence of the fact that $n_1$ and $n_2$ have
  been assumed to be independent). From this, we see that 
  \begin{equation*}
    \expect{\weight(\bfd^{(\ell)}\ldots\bfd^{(0)})}
    =\ell+1-\expect{\zeros(\bfd^{(\ell)}\ldots\bfd^{(0)})}
    =\frac59\ell+O(1), 
  \end{equation*}
  and the results follows by adding the unavoidable $\ell$ squarings.
\end{proof}
\begin{remark}
  This proof can easily be generalised to higher dimensions. In dimension $D$,
  we obtain 
  \begin{equation*}
    \expect{\weight(\bfd^{(\ell)}\ldots\bfd^{(0)})}
    =\biggl(1-\left(\frac23\right)^D\biggr)\ell+O(1).
  \end{equation*}
  On the other hand, the approach in
  \cite{Grabner-Heuberger-Prodinger-Thuswaldner:2005:analy-linear} can be
  generalised to explicitly give the constants now hidden in the error term at
  the cost of more complicated transducers and a delicate analysis of the
  influence of the weight of the standard binary expansion, cf.\
  \cite{Heuberger-Prodinger:2007:hammin-weigh}.
\end{remark}

\section{Errors in \cite{Wu-Lou-Lai-Chang:2007:fast-modul} and
  \cite{Sun-Huai-Sun-Zhang:2007:comput-effic}}\label{sec:errors}
The purpose of this section is to point out where the errors in
\cite{Wu-Lou-Lai-Chang:2007:fast-modul} and
\cite{Sun-Huai-Sun-Zhang:2007:comput-effic} occurred.

The authors of \cite{Wu-Lou-Lai-Chang:2007:fast-modul} write on page~1072 in
Section~4:
\begin{quotation}
  ``Besides, the average proportion of non-zeros in binary representation is
  $\frac12$ and in canonical-signed-digit binary representation is
  $\frac13$. So the average proportion of zeros in the proposed algorithm is
  $(1-\frac12\times\frac13)=\frac56$.''
\end{quotation}
This assertion is erroneous, because the multiplication of $1/2$ and $1/3$
cannot be justified in any way: The weight of a NAF is roughly $1/3$ times the
length of the expansion, not $1/3$ times the weight of the standard binary
expansion. Moreover, the weights of the NAF and the standard binary expansion
are only asymptotically independent,
cf. \cite{Heuberger-Prodinger:2007:hammin-weigh}.

The authors of \cite{Sun-Huai-Sun-Zhang:2007:comput-effic} write on page~1851:
\begin{quotation}
  ``Before the complement recoding, each bit of $E=(e_{k-1}\ldots e_1e_0)_2$
  assumes a value of $0$ or $1$ with equal probability,
  i.e.\ $P(e_i=0)=P(e_i=1)=1/2$ for $0\le i\le k-1$, and there is no dependency
  between any two bits. After the complement recoding, it should be a value of
  $0$ or $1$ with unequal probability, i.e.\ $P(e_i=0)=3/4$ and $P(e_i=1)=1/4$
  for $0\le i\le k-1$. Certainly, there is still no dependency between any two bits.''
\end{quotation}
By construction, there \emph{is} some dependence between the bits, as at most half of
them can be equal to $1$. Next, the claimed probabilities $3/4$ and $1/4$ are
incorrect, this has also been discussed in detail by Yen, Lien and
Moon~\cite{Yen-Lien-Moon:2007:ineff-common} while correcting erroneous claims
in \cite{Chang-Kuo-Lin:2003:fast-algor}.

\section{Conclusions}
We analysed the multi-exponentiation algorithm from
\cite{Wu-Lou-Lai-Chang:2007:fast-modul}. It turns out that the performance
estimates in \cite{Wu-Lou-Lai-Chang:2007:fast-modul} and
\cite{Sun-Huai-Sun-Zhang:2007:comput-effic} are based on incorrect
probabilistic assumptions and are wrong. The method due to
Solinas~\cite{Solinas:2001:low-weigh} or its equivalent formulations have
better performance. Even worse, taking the complement does not have any
positive effect, because the non-adjacent form essentially ignores the
influence of the complement. Thus the proposed algorithm performs as if one
would simply take the NAF for both arguments, which corresponds to the method
proposed by \cite{Dimitrov-Jullien-Miller:2000:compl-fast-algor-multiex}
(without improvements) and is known to be not optimal.



\bibliography{cheub}
\bibliographystyle{amsplain}

\end{document}